\documentclass[11pt,a4paper]{amsart}

\usepackage{amsmath}
\usepackage{amsfonts}
\usepackage{amssymb}
\usepackage{hyperref}
\usepackage{amsthm}
\usepackage{graphicx}
\usepackage{dsfont,fullpage}

\newcommand{\into}{{\hookrightarrow}}
\newcommand{\R}{{\mathbb{R}}}
\newcommand{\C}{{\mathbb{C}}}
\newcommand{\spin}{{\operatorname{spin}}}

\newcommand{\Cl}{{\C l}}
\newcommand{\E}{{\mathcal{E}}}
\newcommand{\F}{{\mathcal{F}}}
\newcommand{\K}{{\mathcal{K}}}
\newcommand{\id}{{\operatorname{id}}}
\newcommand{\bdd}{{\mathfrak{b}}}
\newcommand{\1}{{\mathds{1}}}
\renewcommand{\d}{{\operatorname{d}}}
\newcommand{\N}{{\mathbb{N}}}
\newcommand{\Z}{{\mathbb{Z}}}
\newcommand{\Ann}{{\mathbb{A}}}
\newcommand{\algotimes}{{\otimes_{\text{alg}}}}
\newcommand{\W}{{\mathcal{W}}}
\def\cS{\mathcal S}
\def\S{\mathbb S}
\def\tilde{\widetilde}

\newcommand{\matTwo}[4]{
	{
	\begin{pmatrix}
		#1 & #2 \\
		#3 & #4
	\end{pmatrix}
	}
}
\newcommand{\vecTwo}[2]{
	{
	\begin{pmatrix}
		#1 \\
		#2
	\end{pmatrix}
	}
}
\newcommand{\smallMatTwo}[4]{
	{ \left(
	\begin{smallmatrix}
		#1 & #2 \\
		#3 & #4
	\end{smallmatrix}
	\right) }
}

\DeclareMathOperator{\dom}{dom}
\DeclareMathOperator{\ran}{ran}
\DeclareMathOperator{\Index}{Index}

\newtheorem{theorem}{Theorem}
\newtheorem{proposition}{Proposition}[section]
\newtheorem{corollary}[proposition]{Corollary}
\newtheorem{lemma}[proposition]{Lemma}
\newtheorem{definition}[proposition]{Definition}
\newtheorem{remark}[proposition]{Remark}

\author{Walter D. van Suijlekom}

\address{Institute for Mathematics, Astrophysics and Particle Physics, Radboud University Nijmegen, Heyendaalseweg 135, 6525 AJ Nijmegen, The Netherlands}
  \email{waltervs@math.ru.nl}

  \address{Department of Mathematics, University of Western Ontario, Middlesex College, N6A 5B7 London ON, Canada}
\email{lverhoe@uwo.ca}
  
\author{Luuk S. Verhoeven}
\title{Immersions and the unbounded Kasparov product: embedding spheres into Euclidean space}

\begin{document}


\begin{abstract}
  We construct an unbounded representative for the shriek class associated to the embeddings of spheres into Euclidean space. 
  We equip this unbounded Kasparov cycle with a connection and 
  compute the unbounded Kasparov product with the Dirac operator on $\R^{n+1}$.
  We find that the resulting spectral triple for the algebra $C(\S^n)$ differs from the Dirac operator on the round sphere by a so-called index cycle, 
  whose class 
  in $KK_0(\C, \C)$ represents the multiplicative unit. At all points we check that our construction involving the unbounded Kasparov product is compatible with the bounded Kasparov product using Kucerovsky's criterion and we thus capture the composition law for the shriek map for these immersions at the unbounded KK-theoretical level.
\end{abstract}

\maketitle

\section{Introduction}
\label{sec:intro}
	
	In their 1984 paper on the longitudinal index theorem for foliations \cite{ConnesSkandalis}, Connes and Skandalis prove the wrong-way functoriality of the shriek map.
	The shriek, or wrong-way, map is a class $f_! \in KK(C(X), C(Y))$ associated to a $K$-oriented map $f:X \to Y$ \cite{C82}.
	Indeed, if $f:X \to Y$ and $g:Y \to Z$ we have
	\begin{equation*}
		(g \circ f)_! = f_! \otimes_{C(Y)} g_!
	\end{equation*}
	where $\otimes_{C(Y)}$ denotes the internal Kasparov product over $C(Y)$.
	
	An interesting special case of the shriek map is the fundamental class $[X] \in KK(C(X), \C)$ of a manifold, which is the shriek of the point map $\text{pt}_X:X \to \{*\}$.
	Hence, whenever we have a $K$-oriented map $f:X \to Y$ we get a $KK$-theoretic factorization of fundamental classes
	\begin{equation*}
		[X]= f_! \otimes_{C(Y)} [Y].
	\end{equation*}
	
	This is relevant to noncommutative geometry, since the canonical spectral triple of a manifold \cite{Connes} is an unbounded representative for the fundamental class.
	The construction of $f_!$ given in \cite{ConnesSkandalis} already has a strong unbounded character, so it seems natural to investigate how this factorization of spectral triples can be realized concretely in terms of unbounded KK-cycles in the sense of \cite{BaajJulg}.
	
	When $\pi:M \to B$ is a submersion of compact manifolds, this factorization has already been investigated in \cite{KaadSuijlekom}.
	There a vertical family of Dirac operators $D_\pi$ was constructed, such that the Dirac operator $D_M$ on $M$ decomposes as the following tensor sum
	\begin{equation}
		D_M = D_\pi \otimes 1 + 1 \otimes_\nabla D_B + \kappa,
		\label{eq:submersion_factorization}
	\end{equation}
        in terms of the Dirac operator $D_B$ on the base $B$ lifted to $M$ using a connection $\nabla$, and a bounded operator $\kappa$ which is related to the curvature of $\pi$.

	When the map in question is an immersion $\imath: M \to N$ a similar factorization of Dirac operators should be available. Namely, it should be possible to write the Dirac operator $D_M$ as an unbounded Kasparov product of a shriek element corresponding to $\imath$ and the Dirac operator $D_N$. However, for this to work it is crucial to somehow be able to 
        remove the vertical, or normal, part of the Dirac operator from $D_N$. 
Inspired by the bounded construction here the key ingredient is a Dirac-dual Dirac approach as in \cite{KasparovDual}, see also \cite{Echterhoff}.
	
	In this article we will investigate whether, and how, this factorization works for a simple and concrete set of immersions given by the embeddings $\imath: \S^{n} \into \R^{n+1}$ for $n \geq 1$.
	We start by introducing and constructing the primary ingredients: the unbounded representatives of $ \S^{n}$ and $\R^{n+1}$, and the unbounded shriek cycle of $\imath$ which we also relate to the bounded shriek class $\imath_!$ constructed in \cite{ConnesSkandalis}.
	Next, we investigate the interpretation of the shriek cycle as a dual Dirac, which yields a fourth unbounded $KK$-cycle which we will call the index cycle.
	Its bounded transform --- the so-called index class--- turns out to represent the multiplicative unit in $KK$-theory.
	
	Once we have all ingredients we use a connection on the unbounded shriek cycle to construct a candidate unbounded Kasparov cycle for the product, very much in the spirit of \cite{KaadLeschUnbddProd} and \cite{Mesland}.
	We then use the criterion in \cite{Kucerovsky} to prove that this candidate indeed represents the Kasparov product of $\imath_!$ and $[\R^{n+1}]$ in KK-theory, and that it also represents the product of $[ \S^{n}]$ and the index class, and hence $[ \S^{n}]$ itself.
	This gives the desired factorization of the given immersion $\imath:\S^n \to R^{n+1}$ in terms of the unbounded Kasparov product. 

        \subsection*{Acknowledgements}
          We would like to thank Francesca Arici, Alain Connes, Jens Kaad, Bram Mesland, George Skandalis and Abel Stern for useful discussions and remarks. This research was partially supported by NWO under VIDI-Grant 016.133.326.

\section{The geometry of the spheres in Euclidean space}
\label{sec:geometry}
		From the construction of the shriek class in \cite{ConnesSkandalis} it is clear that the canonical spectral triple of a manifold $M$ represents the fundamental class $[M]$ of that manifold in $KK(C_0(M), \C)$.
		Our first goal is writing the Dirac operator for the embedded spin$^c$ submanifold $ \S^n \subseteq \R^{n+1}$, $n \geq 1$, which of course coincides with the Dirac operator on the round sphere $\S^n$. Then we turn to the unbounded shriek cycle and show that its bounded transform is homotopic to the shriek class in $KK(C(\S^n),C_0(\R^{n+1})$ that was considered in \cite{ConnesSkandalis}. 

	\subsection{Spin geometry of \texorpdfstring{$ \S^{n}$}{Spheres} and \texorpdfstring{$\R^{n+1}$}{Euclidean spaces}}
	\label{sec:spin_geometry_manifolds}
        The first step in the construction of the Dirac operator on the embedded submanifold $\S^n \subseteq \R^n$ is to investigate the $\spin^c$-structure on $ \S^n$ induced by restricting the standard $\spin^c$-structure on $\R^{n+1}$. 	This construction is well known ({\em cf.} \cite{Bures} and \cite{Baer}) but we repeat it here in some detail since later on we will refer to some technical aspects of this construction.

		Let $\imath: \S^n \into \R^{n+1}$ be the standard immersion of the $n$-dimensional sphere into $\R^{n+1}$. 
		Choose some $0 < \varepsilon < 1$ and define a tubular neighbourhood of this immersion by $\tilde{\imath}: \S^n \times (-\varepsilon, \varepsilon) \to \R^{n+1}$ by using geodesic flow along the normal vector field $\partial_r = \frac{1}{r}(x^i \partial_{x^i})$, {\em i.e.} we have in spherical coordinates $\tilde{\imath}(\vec{\theta}, s) = (\vec{\theta}, s+1)$ (see Figure \ref{fig:image_iota}).
		
		\begin{figure}
			\centering
			\includegraphics[width=0.4\textwidth]{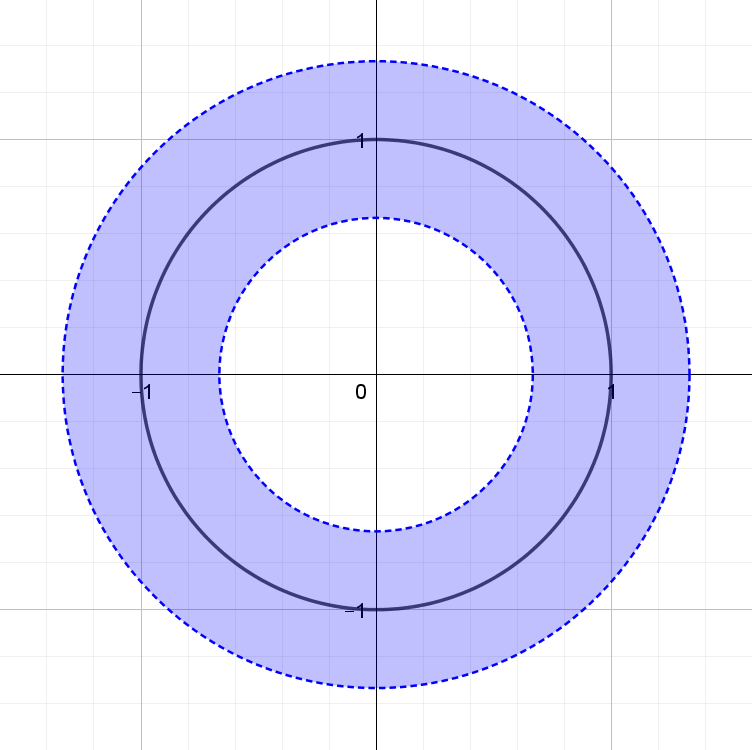}
			\caption{The tubular neighborhood around $\S^n \subseteq \R^{n+1}$ for $n=1$: the black line is the image of $\imath: \S^n \into \R^{n+1}$ and the blue band is the image of $\tilde{\imath}: \S^n \times (-\varepsilon, \varepsilon) \to \R^{n+1}$.}
			\label{fig:image_iota}
		\end{figure}

		Let $\cS$ denote the restriction of the spinor bundle on $\R^{n+1}$ to the image of $\tilde{\imath}$. 
		We can define a Clifford action $\rho$ of $T \S^n$ on $\cS$ by setting $\rho(v) \psi = i c(v) c(\partial_r) \psi$ where $v \in T_x \S^n \subset T_x\R^{n+1}$, $\psi \in \cS_x$ and $c$ denotes the Clifford multiplication on $\R^{n+1}$. We will also write $\gamma_r = c(\partial_r)$.

                In order to describe the induced spinor bundle on $\S^n$ explicitly, we need to distinguish between the odd and even-dimensional case. 

\subsection*{Odd spheres}
		If $n$ is odd, say $n = 2k-1$, the restriction of $\cS$ to $ \S^{2k-1}$ does not immediately yield a spinor bundle. But since in this case $n+1$ is even, $\cS$ has a grading operator $\Gamma$ which decomposes $\cS = \ \cS^+ \oplus \ \cS^-$ into an even and odd part (which are isomorphic).
		The decomposition along $\Gamma$ is preserved by $\rho$, and the restriction $\rho^+$ of $\rho$ to $\ \cS^+$ turns $\ \cS^+$ restricted to $ \S^{2k-1}$ into a spinor bundle on $ \S^{2k-1}$ (\cite{Bures,Baer}).
		
		Using this $\spin^c$ structure on $\S^{2k-1}$ we get a Dirac operator $D_{ \S^{2k-1}}^+$. 
		In accordance to the discussion in Appendix \ref{sect:app} we want to turn this into an even cycle, and we choose to use left-doubling in this case to obtain:
		\begin{equation}
                  (L^2( \S^{2k-1}, \ \cS^+|_{ \S^{2k-1}})\otimes \C^2, \widetilde{ D_{ \S^{2k-1}}}:= D_{ \S^{2k-1}}^+ \otimes \gamma^2; 1 \otimes \gamma^3)
		\end{equation}
		as an even unbounded $C( \S^{2k-1})\otimes \Cl_1$-$\C$ KK-cycle.
		\begin{remark}
Note that equivalently we could have taken $\ \cS^-$ as our defining $\spin^c$ structure, this would have yielded a different Dirac operator $D_{ \S^{2k-1}}^-$.
		Under the isomorphism $\ \cS^+ \cong \ \cS^-$ given by $\gamma_r$ we would have $D_{ \S^{2k-1}}^+ = -D_{ \S^{2k-1}}^-$.
\end{remark}

		The fact that the $\spin^c$ structure on $\S^{2k-1}$ is induced from $\R^{2k}$ allows us to relate the Dirac operators of $\R^{2k}$ and $ \S^{2k-1}$.
		Choosing frames for $\cS$ to identify $\cS\cong \ \cS^+ \otimes \C^2$, $\gamma_r \equiv 1\otimes \gamma^1$ and $\Gamma \equiv 1 \otimes \gamma^3$ we get
		\begin{equation}
			D_{\R^{2k}} = i \frac{1}{r} D^+_{ \S^{2k-1}} \otimes \gamma^2 + i \frac{2k-1}{2r} (1 \otimes \gamma^1) + i \partial_r (1 \otimes \gamma^1).
		\end{equation}
                (see also \cite{Bures} and \cite[Sect. 2]{Baer}).
                %
This represents a class in $KK_0(C_0(\R^{2k+1}), \C)$.

                \subsection*{Even sheres}
		If $n$ is even, say $n=2k$, then $\cS|_{ \S^{2k}}$ immediately yields a spinor bundle on $ \S^{2k}$ which is graded with grading operator $\gamma_r$.
		So the representative for $[ \S^{2k}]$ becomes simply
                $		
                (L^2( \S^{2k}, \cS|_{ \S^{2k}}), D_{ \S^{2k}}; \gamma_r)$.
In this case the relation between the Dirac operator on $\R^{2k+1}$ and $ \S^{2k}$ is given by
		\begin{equation}
			D_{\R^{2k+1}} = i \frac{1}{r} \gamma_r D_{ \S^{2k}} + i \frac{2k}{2r} \gamma_r + i \gamma_r \partial_r.
		\end{equation}
		Finally, the spectral triple representing Euclidean space will be the left-doubled version of the canonical spectral triple:
		\begin{equation*}
                  (L^2(\R^{2k+1}, \cS) \otimes \C^2, \widetilde{D_{\R^{2k+1}}} :=D_{\R^{2k+1}} \otimes \gamma^2; 1 \otimes \gamma^3)
		\end{equation*}
		representing a class in $KK_0(C_0(\R^{2k+1}) \otimes \Cl_1, \C)$.

	\subsection{The shriek class of the immersion}
	\label{sec:immersion_class}
		The class in $KK_1(C( \S^{n}), C_0(\R^{n+1}))$ that we want to associate to $\imath: \S^n \to \R^{n+1}$ is the shriek class, or ``wrong-way'' map.
		We will start by defining an odd unbounded Kasparov $C( \S^{n-1})$-$C_0(\R^{n})$ cycle, and then show in the next subsection that the corresponding bounded transform represents the shriek class as constructed by Connes and Skandalis.
		
		
		Let $\E$ denote the vector space $C_0( \S^n \times (-\varepsilon, \varepsilon))$ and equip it with the $C_0(\R^{n+1})$-valued sesquilinear form
		\begin{equation}
			\langle \psi, \phi \rangle_\E (\vec{\theta}, r) := 
			\begin{cases}
				\frac{1}{r^n} \overline{\psi}(\tilde{\imath}^{-1}(\vec{\theta}, r))\phi(\tilde{\imath}^{-1}(\vec{\theta}, r), &\quad (\vec{\theta}, r) \in \tilde{\imath}( \S^n \times (-\varepsilon, \varepsilon)), \\
				0, &\quad (\vec{\theta}, r) \notin \tilde{\imath}( \S^n \times (-\varepsilon, \varepsilon)).
			\end{cases}
			\label{eq:def_inner_product_on_E}
		\end{equation}
		Furthermore, equip $\E$ with a left- and right-action by $C( \S^n)$ and $C_0(\R^{n+1})$ respectively, by setting
		\begin{equation}
			(g \cdot \psi \cdot h)(\vec{\theta}, s) = g(\vec{\theta}) \psi(\vec{\theta}, s) h(\tilde{\imath}(\vec{\theta}, s)), \label{eq:actions_on_E}
		\end{equation}
		for $g \in C( \S^n)$, $\psi \in \E$ and $h \in C_0(\R^{n+1})$.
		
		\begin{lemma}
		  The sesquilinear form $\langle \cdot, \cdot \rangle_\E$ in Equation \ref{eq:def_inner_product_on_E} turns 
                  $\E$ into a Hilbert $C( \S^n)$-$C_0(\R^{n+1})$ bimodule, with left and right actions as in Equation \ref{eq:actions_on_E}.
			\label{lem:E_is_Hilbert_bimodule}
		\end{lemma}
		\begin{proof}
			The norm induced by $\langle \cdot, \cdot \rangle_\E$ is 
$				\|\psi\|_\E = \left\| \frac{1}{\sqrt{1+s}} \psi \right\|_{\sup}$.
			Since $\frac{1}{\sqrt{1+s}}$ is bounded both from above and away from zero on $(-\varepsilon, \varepsilon)$ this immediately implies that the sesquilinear form is positive definite and that $\E$ is complete.
			The remaining properties are simple verifications.
		\end{proof}
		
		The self-adjoint and regular operator for our candidate unbounded Kasparov cycle representing the shriek class will be the multiplication operator by the function
		\begin{equation*}
			f(s) = \alpha \tan(\alpha s)
		\end{equation*}
		where $\alpha = \frac{\pi}{2\varepsilon}$.
		More precisely, define
		\begin{align}
			& \dom(S) = \{ \psi \in \E | f \psi \in \E \}, \label{eq:definition_of_S}  \qquad (S \psi)(\theta, s) = f(s) \psi(\theta, s). 
		\end{align}
		
		\begin{lemma}
			The operator $S$ defined in Equation \ref{eq:definition_of_S} is self-adjoint, regular and has compact resolvent.
			\label{lem:analysis_of_S}
		\end{lemma}
		\begin{proof}
			For self-adjointness and regularity it suffices to show that $S \pm i$ are surjective.
			Let $\psi \in \E$, then also $\phi := \frac{1}{f \pm i}\psi \in \E$ since $\frac{1}{f+i}$ is in $C_0( \S^{n-1} \times (-\varepsilon, \varepsilon))$.
			Clearly $(S \pm i)\phi = \psi$, hence $S \pm i$ is surjective.
			
			To see that $(S \pm i)^{-1}$ are compact, recall that $\K(C_0(X)) = C_0(X)$ for any locally compact Hausdorff space $X$, where $C_0(X)$ is viewed as a Hilbert module over itself.
			Using the same equivalence of norms we saw in Lemma \ref{lem:E_is_Hilbert_bimodule} we find that $\K(\E) = C_0( \S^{n-1} \times (-\varepsilon, \varepsilon))$, so that $(S \pm i)^{-1}$ is indeed compact.
		\end{proof}
		
		\begin{proposition}
			The data $(\E, S)$ defines an odd unbounded Kasparov cycle between $C( \S^n)$ and $C_0(\R^{n+1})$.
		\end{proposition}
		\begin{proof}
			We know from Lemma \ref{lem:E_is_Hilbert_bimodule} that $\E$ is indeed a Hilbert bimodule between $C( \S^{n-1})$ and $C_0(\R^{n})$ and from Lemma \ref{lem:analysis_of_S} that $S$ has all properties to make $(\E, S)$ into an odd unbounded Kasparov cycle.
		\end{proof}
		
		We now have an unbounded Kasparov cycle, but we want to add one final piece of data.
		Namely, for the purpose of computing the product of this unbounded Kasparov cycle with $[\R^{n+1}]$ 
                we also need a connection on $\E$ relative to $D_{\R^{2k}}$ if $n = 2k-1$ and to $\widetilde{D_{\R^{2k+1}}} = D_{\R^{2k+1}} \otimes \gamma^2$ if $n = 2k$.
		In the following we write, in an abuse of notation, $D_{\R^{n+1}}$ for both $D_{\R^{2k}}$ and $\widetilde{D_{\R^{2k+1}}}$.
		
		\begin{lemma}
			The map
			\begin{equation*}
				\nabla^\E:C_0^1( \S^n \times (-\varepsilon, \varepsilon)) \to \E \otimes_{C_0(\R^{n+1})} \Omega_{D_{\R^{n+1}}}^1
			\end{equation*}
			defined in local spherical coordinates $\vec{\theta} = (\theta^1, ..., \theta^n)$ on $ \S^n$ by 
			\begin{equation*}
				\nabla^\E(\psi) = \left( \frac{\partial \psi}{\partial s} - \frac{n}{2(s+1)}\psi \right) \otimes [D_{\R^{n+1}}, r] + \sum_{i=1}^{n} \frac{\partial \psi}{\partial \theta^i} \otimes [D_{\R^{n+1}}, \theta^i]
			\end{equation*}
                        is a metric connection on $\E$.
			\label{lem:connection}
		\end{lemma}
		\begin{proof}
			The connection property is a straightforward check.
						If we write $\nabla$ for the ``flat'' connection on $\E$, that is, $\nabla^\E$ without the $-\frac{n}{2(s+1)}$ term, and $\langle \cdot, \cdot \rangle$ for the ``flat'' inner product, {\em i.e.} without the factor $\frac{1}{r^n}$, it follows from the fact that $\nabla$ is a metric connection for $\langle \cdot, \cdot \rangle$ that
			\begin{align*}
				\langle \nabla^\E(\psi), \phi \rangle_\E + \langle \psi, \nabla^\E(\phi) \rangle_\E & = \frac{1}{r^n} \langle \nabla(\psi), \phi \rangle + \frac{1}{r^n} \langle \nabla(\psi), \phi \rangle - \frac{n}{r^{n+1}} \langle \psi, \phi \rangle \otimes [D_{\R^{n+1}}, r], \\
				& = \frac{1}{r^n}[D_{\R^{n+1}}, \langle \psi, \phi \rangle] + \left[D_{\R^{n+1}}, \frac{1}{r^n}\right] \langle \psi, \phi \rangle, \\
				& = [D_{\R^{n+1}}, \langle \psi, \phi \rangle_\E].
			\end{align*}
 so that $\nabla^\E$ is a metric connection. 
		\end{proof}
		
		As a final preparation for computing the products we need to use even Kasparov cycles, so we use doubled versions of the index class ({\em cf.} Appendix \ref{sect:app}).
		In the case where $n$ is odd we use left-doubling, so our shriek cycle becomes
		\begin{equation*}
                  (\E \otimes \C^2, S \otimes \gamma^2; 1 \otimes \gamma^3; \nabla \otimes \id)
		\end{equation*}
                representing a class in $KK_0(C( \S^n) \otimes \Cl_1, C(\R^{n+1}))$. When $n$ is even we use right-doubling, which makes our shriek cycle
		\begin{equation*}
                  (\E \otimes \C^2, S \otimes \gamma^1; 1 \otimes \gamma^3; \nabla \otimes \id),
		\end{equation*}
           this time representing a class in $KK_0(C( \S^n), C(\R^{n+1}) \otimes \Cl_1)$. In both cases we will denote the unbounded KK-cycle by $(\tilde \E,\tilde S)$ and the shriek cycle $(\E,\bdd(S))$ obtained in KK-theory by bounded transform by $\imath_!$. 
		
	\subsection{Equivalence to bounded construction}
		
	We will now 
        show that the bounded transform $(\E, \bdd(S))$ is homotopic to the shriek cycle as constructed in \cite{ConnesSkandalis}.
		This in fact already proves the factorization $[ \S^n] = \imath_! \otimes [\R^{n+1}]$ as $KK$-classes, but we want to prove this factorization in full geometric detail in the unbounded KK-theoretic context. 
		
		In \cite{ConnesSkandalis} one allows any map $\tilde{\imath}_{CS}: \S^n \times \R \to \R^{n+1}$ which is a diffeomorphism onto a tubular neighbourhood of $\imath( \S^n) \subset \R^{n+1}$.
		For our purposes we choose $\tilde{\imath}_{CS}$ such that $\tilde{\imath}_{CS}|_{ \S^n \times (-\varepsilon, \varepsilon)} \equiv \tilde{\imath}$.
		
		One defines a $C_0(\R^{n+1})$-valued sesquilinear form $\langle \cdot, \cdot \rangle_{CS}$ on $C_c( \S^n \times \R)$ by setting
		\begin{align*}
			\langle \psi, \phi \rangle_{CS}(x) = \overline{\psi}(\tilde{\imath}_{CS}^{-1}(x))\phi(\tilde{\imath}_{CS}^{-1}(x))
		\end{align*}
		for $x$ in the tubular neighbourhood, and $\langle \psi, \phi \rangle_{CS} = 0$ elsewhere.
		There is a left $C( \S^n)$ action and a right $C_0(\R^{n+1})$ action on $C_c(\S^n \times \R)$ given by
		\begin{align*}
			(g \cdot \psi \cdot h)(\vec{\theta}, s) = g(\vec{\theta})\psi(\vec{\theta}, s)h(\tilde{\imath}_{CS}(\vec{\theta}, s)).
		\end{align*}
		This turns $C_c( \S^n \times \R)$ into a pre-Hilbert bimodule; denote by $\E_{CS}$ the corresponding Hilbert $C( \S^n)$-$C_0(\R^{n+1})$-bimodule.
		It is easy to see that $\E_{CS} = C_0( \S^n \times \R)$.
		
		Next, one chooses a function $M:[0, \infty) \to [0, 1]$ such that $M(0) = 1$ and $M$ has compact support.
		On $\E_{CS}$ define an operator $F:\E_{CS} \to \E_{CS}$ by
		\begin{align*}
			(F \psi)(\theta, s) = \sqrt{1 - M(|s|)} \frac{s}{|s|} \psi(\theta, s).
		\end{align*}
		
		For instance, we may choose
		\begin{align*}
			M(s) = 
				\begin{cases}
					\frac{1}{1+f(s)^2} & \quad s \in [0, \varepsilon), \\
					0 & \quad s \geq \varepsilon,
				\end{cases}
		\end{align*}
		so that
		\begin{align*}
			(F \psi)(\theta, s) = 
				\begin{cases}
					-\psi(\theta, s) & \quad s \leq -\varepsilon, \\
					\frac{f(s)}{\sqrt{1+f(s)^2}}\psi(\theta, s) & \quad s \in (-\varepsilon, \varepsilon), \\
					\psi(\theta, s) & \quad s \geq \varepsilon.
				\end{cases}
		\end{align*}
		
		This already closely resembles $(\E, \bdd(S))$, the major difference is that $\E$ uses $(-\varepsilon, \varepsilon)$ as fibre with an operator tending to 1 at the edge, while $\E_{CS}$ uses $\R$ as fibre with an operator that equals 1 outside $(-\varepsilon, \varepsilon)$. However, the two cycles represent the same class in KK-theory because of the following result. 
                \begin{proposition}
The two bounded KK-cycles $(\E,\bdd(S))$ and $(\E,F)$ are homotopic and, consequently, they define the same class in $KK(C( \S^n), C_0(\R^{n+1}))$.
                  \end{proposition}
\proof              
We will construct a bounded Kasparov cycle 
between $C( \S^n)$ and $C_0(\R^{n+1})\otimes C([0,1])$ such that evaluation at 0 yields a cycle unitarily equivalent to $(\E, \bdd(S))$ and evaluation at 1 yields a cycle equivalent to $(\E_{CS}, F)$. 
		
		Let $R:[0,1) \to \R$ be any increasing function such that $R(0) = \varepsilon$ and $R(x) \to \infty$ as $x \to 1$.
		Define $X \subset  \S^n \times \R \times [0,1]$ by $(\vec{\theta}, s, t) \in X$ if $t = 1$ or $|s| < R(t)$ for $t < 1$.
		
		Set $\F = C_0(X)$, and define a $C_0(\R^{n+1})\otimes C([0,1]) = C_0(\R^{n+1} \times [0,1])$-valued sesquilinear form on $\F$ by
		\begin{align*}
			\langle \psi, \phi \rangle_{CS}(\vec{\theta}, r, t) = 
			\begin{cases}
				\overline{\psi}(\tilde{\imath}^{-1}(\vec{\theta}, r), t)\phi(\tilde{\imath}^{-1}(\vec{\theta}, r), t) &\, (\vec{\theta}, r) \in \tilde{\imath}( \S^n \times (-R(t), R(t))), \\
				0, &\, (\vec{\theta}, r) \notin \tilde{\imath}( \S^n \times (-R(t), R(t))).
			\end{cases}
		\end{align*}
		We may equip $\F$ with a left-$C( \S^n)$ and right-$C_0(\R^{n+1} \times [0,1])$ module structure by setting
		\begin{align*}
			(f \cdot \psi \cdot g)(\vec{\theta}, s, t) = f(\vec{\theta})\psi(\vec{\theta}, s, t)g(\tilde{\imath}_{CS}(\vec{\theta}, s), t).
		\end{align*}
		Note that the norm on $\F$ induced by this inner product is simply the $\sup$-norm on $C_0(X)$, so that $\F$ is indeed a Hilbert bimodule.
		Then $\mathcal{L}(\F) = C_b(X)$ and $\K(\F) = C_0(X)$.
		
		Now we define an operator $G$ on $\F$ by
		\begin{align*}
			(G \psi)(\theta, s, t) = 
				\left\{
				\begin{array}{l l}
					-\psi(\theta, s, t), & s \leq -\varepsilon \\
					\frac{f(s)}{\sqrt{1+f(s)^2}}\psi(\theta, s, t), & s \in (-\varepsilon, \varepsilon) \\
					\psi(\theta, s, t). & s \geq \varepsilon
				\end{array}
			\right.
		\end{align*}
		Note that $G^2 - 1$ is in $\K(\F)$ since it is in $C_0(X)$.

                We claim that $(\F, G)$ is a homotopy between $(\E, \bdd(S))$ and $(\E_{CX}, F)$. Indeed, for $i = 0,1$ we denote by $B_i$ the Hilbert bimodule corresponding to the $C^*$-homomorphism $\phi_i:C_0(\R^{n+1} \times [0,1]) \to C_0(\R^{n+1})$, $\phi(f)(\theta, r) = f(\theta, r, i)$.
			For the evaluation at $t = 0$ the map
			\begin{align*}
				& U:\F \otimes_{C_0(\R^{n+1} \times [0,1])} B_0 \to \E \\
				& U(\psi \otimes g)(\vec{\theta}, s) = (s + 1)^{\frac{n}{2}}\psi(\vec{\theta}, s, 0)g(\tilde{\imath}(\vec{\theta}, s)).
			\end{align*}
			is a unitary equivalence between $(\F \otimes_{C_0(\R^{n+1} \times [0,1])} B_0, G \otimes 1)$ and $(\E, \bdd(S))$.
		
			At $t = 1$ the map
			\begin{align*}
				& V:\F \otimes_{C_0(\R^{n+1} \times [0,1])} B_1 \to \E_{CS} \\
				& V(\psi \otimes g)(\vec{\theta}, s) = \psi(\vec{\theta}, s, 1)g(\tilde{\imath}_{CS}(\vec{\theta}, s)).
			\end{align*}
			is a unitary equivalence between $(\F \otimes_{C_0(\R^{n+1} \times [0,1])} B_1, G \otimes 1)$ and $(\E_{CS}, F)$.
		\endproof
		
			
	
	\subsection{The index class}
	\label{sec:index_class}
		
		In our sought-for KK-factorization of $D_{\S^n}$ in terms of $D_{\R^{n+1}}$, the cycle $(\tilde{\E}, \tilde{S})$ should in some way cancel out the normal, or radial, direction.
		This dimension reduction is, in bounded $KK$-theory, accomplished by a dual-Dirac element, as in \cite{Echterhoff}.
		In our case $\tilde{S}$ is expected to act as an unbounded dual-Dirac element, and this leads us to investigate the interaction between the radial derivative in $D_{\R^{n+1}}$ and the radial function defining $\tilde{S}$.
		
		So, let us define a symmetric operator $T_0$ on $\dom(T_0) = C_c^\infty((-\varepsilon, \varepsilon), \C^2) \subset L^2((-\varepsilon, \varepsilon), \C^2)$ by 
		\begin{equation*}
			T_0 \psi = i\gamma^1 \partial_s \psi + \gamma^2 f(s)\psi = \matTwo{0}{i\partial_s - if(s)}{i\partial_s + if(s)}{0}\psi,
		\end{equation*}
		where $f(s) = \alpha \tan(\alpha s)$ and $\alpha = \frac{\pi}{2\varepsilon}$ as before.
		
		We want to show that the closure $T:= \overline{T_0}$, together with Hilbert space $L^2((-\varepsilon, \varepsilon), \C^2)$ and grading $\gamma^3$, defines an even Kasparov $\C$-$\C$ cycle, and that this cycle represents the multiplicative unit in $KK_0(\C, \C)$.
		We will refer to $(L^2((-\varepsilon, \varepsilon), \C^2), T; \gamma^3)$ as the {\em index cycle}, and to the corresponding $KK_0(\C, \C)$ class as the {\em index class}, which we denote by $\1$.
		
		In order to prove essential self-adjointness of $T_0$, we first find integrating factors $I$ and $J$ for the differential equation $(T_0 + \lambda i)u = g$.
		We then use these integrating factors to show that $\ran(T_0 + \lambda i)$ is the ``orthogonal complement'' of $J$ and finally we show that this is dense.
		This argument is based on \cite[Example 33.1]{Lax}.
		
		\begin{lemma}
			\label{lem:I_J_integrating_factors}
			Suppose $u, g \in C_c^\infty((-\varepsilon, \varepsilon), \C^2)$ and $\lambda^2 = \alpha^2$, then
			\begin{equation*}
				\frac{\d}{\d x} I_\lambda u = J_\lambda g
			\end{equation*}
			if and only if $g = (T_0 + \lambda i)u$, for
			\begin{equation*}
				I_\lambda(s) = \matTwo{1 + s f(s)}{\lambda s}{\frac{1}{\lambda}f(s)}{1}, \,
				J_\lambda(s) = -i\matTwo{\lambda s}{1 + s f(s)}{1}{\frac{1}{\lambda}f(s)}. 
			\end{equation*}
		\end{lemma}
		\begin{proof}

			Using the differential equation
			\begin{equation*}
				f(s)^2 - f'(s) + \alpha^2 = 0
			\end{equation*}
			which is satisfied by $f(s) = \alpha \tan(\alpha s)$ it is straightforward to show that
			\begin{equation*}
				J_\lambda^{-1} \frac{\d }{\d x} I_\lambda = (T_0 + \lambda i). \qedhere
			\end{equation*}

		\end{proof}
		
		The next step is to show that the range of $T_0 + \lambda i$ is the ``orthogonal complement'' of $J_\lambda$ in $C_c^\infty((-\varepsilon, \varepsilon), \C)$.
		
		\begin{lemma}
			\label{lem:range_of_T0}
			For $\lambda = \pm \alpha$ and $J_\lambda$ as in Lemma \ref{lem:I_J_integrating_factors} we have
			\begin{align*}
				\ran(T_0 + \lambda i) = \left\{ g \in C_c^\infty((-\varepsilon, \varepsilon), \C^2) \, \middle| \, \int_{-\varepsilon}^\varepsilon J(x)g(x) \d x = 0 \right\}.
			\end{align*}
		\end{lemma}
		\begin{proof}
			Suppose $g = (T_0 + \lambda i)u$, then by Lemma \ref{lem:I_J_integrating_factors}
			\begin{align*}
				\int_{-\varepsilon}^\varepsilon J(x)g(x) \d x & = \int_{-\varepsilon}^\varepsilon \left(\frac{\d}{\d x} I(x)u(x)\right) \d x = 0,
			\end{align*}
			since $u \in C_c^\infty((-\varepsilon, \varepsilon), \C^2)$.
			Also, $g$ is indeed in $C_c^\infty((-\varepsilon, \varepsilon), \C^2)$.
			
			For the converse, suppose $g \in C_c^\infty((-\varepsilon, \varepsilon), \C^2)$ such that $\int Jg = 0$.
			Define
			\begin{align*}
				u(x) = I^{-1}(x) \int_{-\varepsilon}^x J(y)g(y) \d y,
			\end{align*}
			then certainly $u \in C^\infty_c((-\varepsilon, \varepsilon), \C^2)$, and
			\begin{align*}
				\frac{\d}{\d x} I(x) u(x) = J(x) g(x).
			\end{align*}
			Then by Lemma \ref{lem:I_J_integrating_factors} we have $(T_0 + \lambda i)u = g$.
		\end{proof}
		
		Finally we want to show that the range of $T_0 + \lambda i$ is dense for $\lambda = \pm \alpha$.
		The intuition here is that $J_\lambda$ is ``not $L^2$'', so that the ``orthogonal complement'' of $J_\lambda$ is dense. More precisely, we have the following result. 
		
		\begin{lemma}
			\label{lem:orth_complement_non_L2_dense}
			Let $\Omega \subset \R^d$ and $j \in C(\Omega, \C^n)$, $j \notin L^2(\Omega, \C^n)$.
			Then 
			\begin{equation*}
				K_j = \left\{ g \in C_c^\infty(\Omega, \C^n) \middle| \int_\Omega \langle j(x), g(x) \rangle \d x = 0 \right\}
			\end{equation*}
			is dense in $L^2(\Omega, \C^n)$.
		\end{lemma}
		\begin{proof}
			Define a linear functional $\langle j | : C_c^\infty(\Omega) \to \C$ by $\langle j | f = \int_\Omega \langle j(x), f(x) \rangle \d x$.
			Our first step is to prove that $\langle j |$ is unbounded.
			
			Suppose $\langle j |$ were bounded on $C_c^\infty(\Omega, \C^n)$ with respect to the $L^2(\Omega, \C^n)$-norm.
			Then $\langle j|$ extends to a bounded linear functional on $L^2(\Omega, \C^n)$, given by $\psi \mapsto \langle \tilde{j}, \psi \rangle$ for some $\tilde{j} \in L^2(\Omega, \C^n)$ by Riesz-representation. But then $\int_\Omega \langle j(x), g(x) \rangle \d x = \int_\Omega \langle \tilde{j}(x), g(x) \rangle$ for all $g \in C_c^\infty(\Omega)$, which implies $j(x) = \tilde{j}(x)$.
			This is in contradiction with our assumption that $j \notin L^2(\Omega, \C^n)$.
			
			So $\langle j |$ is an unbounded linear functional on $C_c^\infty(\Omega, \C^n)$.
			Therefore there exists a sequence $(\delta_m)_{m \in \N}$ in $C_c^\infty(\Omega, \C^n)$ such that $\langle j| \delta_m = 1$ and $\|\delta_m \|_{L^2} < \frac{1}{m}$ for all $m \in \N$.
			
			Let $\psi \in L^2(\Omega, \C^n)$ and $\epsilon > 0$ be arbitrary.
			Then there is an $\psi_1 \in C_c^\infty(\Omega, \C^n)$ such that $\|\psi - \psi_1\|_{L^2} < \frac{1}{2}\epsilon$.
			Define $\alpha = \langle j | \psi_1$ and find $M$ such that $\frac{\alpha}{M} < \frac{1}{2}\epsilon$.
			Set $\psi_2 = \psi_1 - \alpha \delta_M$, then $\|\psi - \psi_2\|_{L^2} < \varepsilon$ and $\langle j | \psi_2 = 0$, proving density of $K_j$.
		\end{proof}
		
		\begin{proposition}
			\label{prop:range_of_T0_dense}
			The range of $T_0 + \lambda i$ is dense for $\lambda = \pm \alpha$. Consequently, $T_0$ is essentially self-adjoint. 
		\end{proposition}
		\begin{proof}
			Write $j_1$ and $j_2$ for the rows of $J_\lambda$, so
			\begin{align*}
				j_1(s) = i\vecTwo{\lambda s}{1 + sf(s)}, \quad
				j_2(s) = -i\vecTwo{1}{\frac{1}{\lambda}f(s)}.
			\end{align*}
			Then Lemma \ref{lem:range_of_T0} tells us that the range of $T_0 + \lambda i$ is $K_{j_1} \cap K_{j_2}$, in the notation of Lemma \ref{lem:orth_complement_non_L2_dense}.
			
			To prove density of $K_{j_1} \cap K_{j_2}$ we use the same strategy as in Lemma \ref{lem:orth_complement_non_L2_dense} to obtain two sequences $(\delta^1_m)_{m \in \N}$ and $(\delta^2_m)_{m \in \N}$ such that $\langle j_i | \delta^i_m = 1$ and $\|\delta^i_m\|_{L^2} < \frac{1}{m}$.
			
			Write $\delta^i_{m, 1}$ and $\delta^i_{m, 2}$ for the first and second components of $\delta^i_m$ respectively, and $\tau:(-\varepsilon, \varepsilon) \to (-\varepsilon, \varepsilon)$, $\tau(x) = -x$.
			Since the first component of $j_1$ is odd, while the second component is even we may replace $\delta^1_m$ by
			\begin{align*}
				\tilde{\delta}^1_m = \frac{1}{2}\vecTwo{\delta^1_{m,1} - \delta^1_{m,1} \circ \tau}{\delta^1_{m,2} + \delta^1_{m,2} \circ \tau}.
			\end{align*}
			On the other hand, the first component of $j_2$ is even, while the second component is odd, so we may replace $\delta^2_m$ by
			\begin{align*}
				\tilde{\delta}^2_m = \frac{1}{2}\vecTwo{\delta^2_{m,1} + \delta^2_{m,1} \circ \tau}{\delta^2_{m,2} - \delta^2_{m,2} \circ \tau}.
			\end{align*}
			Replacing $\delta^i_m$ by $\tilde{\delta}^i_m$ does not change the values of $\langle j_i| \delta^i_m$, and it does not increase the norm of the $\delta^i_m$.
			Furthermore, since the corresponding components of $j_i$ and $\tilde{\delta}^i_m$ now have opposite parity $\langle j_1 | \tilde{\delta}^2_m = \langle j_2 | \tilde{\delta}^1_m = 0$.
			
			We can now complete the density proof similar to the final step in Lemma \ref{lem:orth_complement_non_L2_dense}.
			Let $\psi \in L^2((-\varepsilon, \varepsilon), \C^2)$ and $\epsilon > 0$ be arbitrary.
			Then there is a $\psi_1 \in C_c^\infty((-\varepsilon, \varepsilon), \C^2)$ such that $\|\psi - \psi_1\|_{L^2} < \frac{1}{3}\epsilon$.
			Let $\alpha_i = \langle j_i | \psi_1$ for $i = 1, 2$ and find $M$ such that $\frac{\alpha_i}{M} < \frac{1}{3}\epsilon$.
			Then $\psi_2 = \psi_1 - \alpha_1 \tilde{\delta}^1_M - \alpha_2 \tilde{\delta}^2_M$ satisfies both $\langle j_1|\psi_2 = 0$, $\langle j_2|\psi_2 = 0$ and $\|\psi - \psi_2\|_{L^2} < \epsilon$.
		\end{proof}
			
		The other property of $T$ that we need is that of compact resolvent. Its proof is based on the following result. 
		
		\begin{lemma}
			\label{lem:graph_norm_T}
			The graph-norm of $T \pm \lambda i$ is larger than the Sobolev norm for $\lambda^2 \geq \alpha^2$. Indeed, for $\psi \in C_c^\infty((-\varepsilon, \varepsilon), \C^2)$ we have $\|\psi\|^2 + \|(T + i\lambda) \psi\|^2 > \|\psi\|^2 + \|\psi'\|^2$.
		\end{lemma}
		\begin{proof}
			We want to compute $\|(T + i\lambda)\psi\|^2$ for $\psi \in C_c^\infty((-\varepsilon, \varepsilon), \C^2)$, the domain of $T_0$.
			The claim then follows for $T$ by continuity.
			Using the symmetry of $T$ this equals $\langle \psi, (T^2 + \lambda^2)\psi \rangle$, so let us compute $T^2$.
			\begin{align*}
				T^2 = \matTwo{-\partial_s^2 - f'(s) + f(s)^2}{0}{0}{-\partial_s^2 + f'(s) + f(s)^2}.
			\end{align*}
			Therefore
			\begin{align*}
				\langle \psi, (T^2 + \lambda^2)\psi \rangle = \langle \psi, -\psi'' \rangle + \langle \psi, \matTwo{f(s)^2 - f'(s) + \lambda^2}{0}{0}{f(s)^2 + f'(s) + \lambda^2}\psi \rangle.
			\end{align*}
			For $\lambda^2 \geq \alpha^2$ both $f(s)^2 \pm f'(s) + \lambda^2 \geq 0$, so the second term on the right-hand-side is positive.
			Hence
			\begin{align*}
				\langle \psi, (T^2 + \lambda^2) \psi \rangle \geq \langle \psi, - \psi'' \rangle.
			\end{align*}
			By partial integration $\langle \psi, -\psi'' \rangle = \langle \psi', \psi' \rangle$ so we find that
\[
				\|(T + \lambda i)\psi\|^2 \geq \|\psi'\|^2.\qedhere
\]
		\end{proof}
		\begin{corollary}
			\label{cor:domT_in_Sobolev}
			The domain of $T$ is contained in the first-order Sobolev space $H^1((-\varepsilon, \varepsilon), \C^2)$.
		\end{corollary}
		
		\begin{proposition}
			\label{prop:T_compact_resolvent}
			The resolvent $(T + \lambda i)^{-1}$ is compact for $\lambda = \pm \alpha$ and hence for all $\lambda \in \rho(T)$.
		\end{proposition}
		\begin{proof}
			Define $D = \left\{ \psi \in L^2((-\varepsilon, \varepsilon), \C^2) \middle| \, \|\psi\| \leq 1 \right\}$ the unit disc in $L^2$.
			We will prove that $M:= (T + \lambda i)^{-1}D$ is pre-compact.
			
			Let $\psi = (T + \lambda i)^{-1}\phi$, $\phi \in D$.
			Then $\|\psi\| \leq |\lambda|^{-1}$ since $\|(T + \lambda i)^{-1}\| \leq |\lambda|^{-1}$ and $\psi \in \dom(T) \subset H^1((-\varepsilon, \varepsilon), \C^2)$ by Corollary \ref{cor:domT_in_Sobolev}.
			Furthermore Lemma \ref{lem:graph_norm_T} tells us that 
			\begin{align*}
				\|\psi'\| \leq \|(T + \lambda i) \psi\| = \|\phi\| \leq 1.				
			\end{align*}			
			Therefore
			\begin{align*}
				M \subset \left\{ \psi \in H^1((-\varepsilon, \varepsilon), \C^2) \middle| \, \|\psi'\|, \|\psi\| \leq \max(1, |\lambda|^{-1}) \right\}.
			\end{align*}
			
			By the Rellich embedding theorem the set on the right hand side is compact, so that $M$ is pre-compact.
			Compactness of the resolvents for $\lambda \neq \pm \alpha$ follows from the first resolvent identity $(T + \lambda i)^{-1} = (T + \alpha i)^{-1} + (\lambda - \alpha)(T + \lambda i)^{-1}(T + \alpha i)^{-1}$.
		\end{proof}
		
		\begin{remark}
			The operator $T^2 + \lambda^2$ is actually a Schr\"odinger type operator on $L^2((-\varepsilon, \varepsilon), \C^2)$, which for $\lambda$ large enough has positive potential.
			It is a classical result that Schr\"odinger operators with bounded potential on a bounded domain and Schr\"odinger operators on an unbounded domain with a confining potential have compact resolvents.
			The reason we did not use these classical results is that we are dealing with a combined case here: while $f(s)^2 - f'(s) + \lambda^2$ is a bounded potential, $f(s)^2 + f'(s) + \lambda^2$ is unbounded (it is, however, confining).
			Therefore we have provided a direct proof along the lines of proofs for Schr\"odinger operators as found in \cite{ReedSimon}.
		\end{remark}

		\begin{proposition}
			\label{prop:index_T_1}
			The data $(\C, L^2((-\varepsilon, \varepsilon), \C^2), T; \gamma^3)$ is an even spectral triple that represents the multiplicative unit in $KK_0(\C, \C)$.
		\end{proposition}
		\begin{proof}
                  We have already showed that $T$ is self-adjoint and has compact resolvents. 
                  Also 
                  we have $T \gamma^3 = -\gamma^3 T$ so that 
                  $(\C, L^2((-\varepsilon, \varepsilon), \C^2), T; \gamma^3)$ is an even spectral triple. So, since $\Index:KK_0(\C, \C) \to \Z$ is an isomorphism of rings, 
it suffices to show that $\Index T = 1$ to conclude that $(L^2((-\varepsilon, \varepsilon), \C^2), T)$ is an unbounded representative for the multiplicative unit in $KK_0(\C, \C)$.
			
			Write $T_+ = i\partial_s + if(s)$ and $T_- = i\partial_s - if(s) = T_+^*$ so that $T = \smallMatTwo{0}{T_+}{T_-}{0}$. Let us then compute the index of $T$. 
			First of all $u \in \ker T_+$ if and only if $u$ satisfies the differential equation
			\begin{align*}
				0 = iu'(s) + if(s)u(s).
			\end{align*}
			This is a first-order, one dimensional ODE so all solutions are given by
			\begin{align*}
				u(s) = C e^{-F(s)}
			\end{align*}
			for $C \in \C$ and $F$ a primitive function for $f$.		
			But a primitive function for $f(s) = \alpha \tan(\alpha s)$ is $F(s) = -\ln(\cos(\alpha s))$, so the kernel of $T_+$ is given by constant multiples of $u_+(s) = \cos(\alpha s)$, so $\ker T_+ = \C u_+$.
			Similarly we find that the kernel of $T_-$ is given by constant multiples of $u_-(s) = \cos(\alpha s)^{-1}$. However, $u_-$ is not an $L^2(-\varepsilon, \varepsilon)$ function so $\ker T_- = \{0\}$.
			Hence $\Index(T) = \dim\ker(T_+) - \dim\ker(T_-) = 1$.
		\end{proof}
		
\section{Kasparov product of the shriek cycle with the plane}
\label{sec:product_immersion_plane}
	In the spirit of \cite{KaadLeschUnbddProd} and \cite{Mesland} we can use the connection on $(\tilde{\E}, \tilde{S})$ to construct a candidate unbounded Kasparov cycle for the product $\imath_! \otimes [\R^{n+1}]$.

	We will begin by considering the product of the Hilbert bimodules $\tilde{\E} \otimes_{C_0(\R^{n+1})} L^2(\R^{n+1}, \cS)$, followed by the computation of the product operator 
	\begin{equation}
		D_\times = \tilde{S} \otimes 1 + \gamma^3 \otimes_\nabla D_{\R^{n+1}}
		\label{eq:product_operator_definition}
	\end{equation}
        with domain $\dom(\tilde S)\algotimes \dom(D_{\R^{n+1}})$, where we still use the notation $D_{\R^{n+1}}$ for $D_{\R^{2k}}$ and $\widetilde{D_{\R^{2k+1}}}$.
	We will then prove that the operator $D_\times$ on the balanced tensor product of $\widetilde \E$ and $L^2(\R^{n+1}, \cS)$ is an unbounded Kasparov cycle and that it represents not only $\imath_! \otimes [\R^{n+1}]$ but also the fundamental class $[ \S^n]$.
		
	\subsection{Computation of the unbounded Kasparov product }
	\label{sec:product_immersion_plane_computations}
	
		The motivation for including the factor $\frac{1}{r^n}$ in $\langle \cdot, \cdot \rangle_\E$ was to ``flatten'' a neighbourhood of the circle in $\R^{n+1}$ to a cylinder.
		This is indeed accomplished, as we see in the computation of the balanced tensor product.
		
		\begin{proposition}
                  Let $T$ be the index cycle defined in Section \ref{sec:index_class}. Then we have for the odd and even-dimensional spheres that
			\label{lem:product_space_of_E_R2}
			\begin{description}
				\item[$n=2k+1$ ]
				There is a unitary isomorphism $U$ from the Hilbert bimodule $\tilde{\E} \otimes_{C_0(\R^{2k})} L^2(\R^{2k}, \cS)$ to $L^2( \S^{2k-1}, \ \cS^+ \otimes \C^2) \otimes L^2((-\varepsilon, \varepsilon), \C^2)$ such that
				\begin{equation}
					U D_\times U^* = \widetilde{D_{ \S^{2k-1}}} \otimes \frac{1}{1+s} + \gamma_3 \otimes T. \label{eq:product_operator_radial_split}
				\end{equation}
				\item[$n=2k$ ]
				There is a unitary isomorphism $U$ from the Hilbert bimodule $\tilde{\E} \otimes_{C_0(\R^{2k+1}) \otimes \Cl_1} (L^2(\R^{2k+1}, \cS) \otimes \C^2)$ to $L^2( \S^{2k}, \cS) \otimes L^2((-\varepsilon, \varepsilon), \C^2)$ such that
				\begin{equation*}
					U D_\times U^* = D_{ \S^n} \otimes \frac{1}{1+s} + \gamma_r \otimes T.
				\end{equation*}
			\end{description}
		\end{proposition}
		\begin{proof}
			This proof will be done for $n$ odd, the same strategy works for $n$ even. 
			We will build the unitary equivalence in several steps, starting from the unitary map
			\begin{align*}
				& V:\E \otimes_{C_0(\R^{2k})} L^2(\R^{2k}, \cS) \to L^2( \S^{2k-1} \times (-\varepsilon, \varepsilon), \cS) \otimes \C^2, \\
				& V(g \otimes \psi)(\vec{\theta}, s) = g(\vec{\theta}, s)\psi(\vec{\theta}, s + 1).
			\end{align*}
			Let us first check that this is actually a unitary map.
			\begin{align*}
				\langle g \otimes \psi, g' \otimes \psi' \rangle_{\E \otimes L^2} & = \langle \psi, \langle g, g' \rangle_\E \cdot \psi' \rangle_{L^2(\R^{2k}, \cS)}, \\
				& = \int_{\R^{2k}} \langle \psi(\theta, r), \langle g, g' \rangle_{\E}(\theta, r) \,\psi'(\theta, r) \rangle_\cS r^{2k-1} \d r\d\theta, \\
				& = \int_{\Ann^{2k}} \langle \psi(\theta, r), \frac{1}{r^{2k-1}}\overline{g(\theta, r-1)}g'(\theta, r-1)\psi'(\theta, r)\rangle_\cS \, r^{2k-1} \d r \d \theta, \\
				& = \int_{ \S^{2k-1} \times (-\varepsilon, \varepsilon)} \langle g(\theta, s)\psi(\theta, s+1), g'(\theta, s)\psi'(\theta, s+1) \rangle_\cS \, \d s \d \theta, \\
				& = \langle V(g \otimes \psi), V(g' \otimes \psi') \rangle_{L^2( \S^{2k-1} \times (-\varepsilon,\varepsilon), \cS)}.
			\end{align*}
			Furthermore $V$ is surjective since $\E$ contains an approximate identity for $L^2( \S^{2k-1} \times (-\varepsilon, \varepsilon), \cS)$ consisting of bump functions with growing support.
			
			We now apply the equivalence $\cS \cong \ \cS^+ \otimes \C^2$ while moving from $\E$ to $\tilde{\E} = \E \otimes \C^2$ to obtain a unitary equivalence from $\tilde{\E} \otimes_{C_0(\R^{2k})} L^2(\R^{2k}, \cS)$ to $L^2( \S^{2k-1} \times (-\varepsilon, \varepsilon), \ \cS^+) \otimes \C^2 \otimes \C^2$.
			Note that the grading on this space is given by $1 \otimes \gamma_3 \otimes \gamma_3$.
			
			Under this unitary equivalence, the operator 
                        $D_\times$ transforms as follows. 
			The term $\tilde{S} \otimes 1$ simply becomes $f(s) \otimes 1 \otimes \gamma_2$, while 
                        $\gamma_3 \otimes_{\nabla^{\tilde{\E}}} D_{\R^{2k}}$ transforms to 
			\begin{equation*}
				\gamma_3 \otimes_{\nabla^{\tilde{\E}}} D_{\R^{2k}} \sim \frac{1}{r}D^+_{ \S^{2k-1}} \otimes \gamma_2 \otimes \gamma_3 + i \partial_s \otimes \gamma_1 \otimes \gamma_3,
			\end{equation*}
			where there is a crucial cancellation between the term $i \frac{2k-1}{2r}$ from $D_{\R^{2k}}$ against the $-\frac{2k-1}{2(s+1)}$ in the connection.
			
			We now apply the following unitary transformation to the $\C^2 \otimes \C^2$ component.
			\begin{equation*}
				W = \frac{1}{\sqrt{2}}(1 \otimes \gamma_3 + \gamma_2 \otimes \gamma_2)
			\end{equation*}
			The properties of the $\gamma$-matrices make it straightforward to check that $W$ is a unitary such that
			\begin{align*}
				W(\gamma_3 \otimes \gamma_3)W^* = \gamma_3 \otimes \gamma_3, & &
				W(1 \otimes \gamma_2)W^* = \gamma_2 \otimes \gamma_3, \\
				W(\gamma_1 \otimes \gamma_3)W^* = \gamma_1 \otimes \gamma_3, & &
				W(\gamma_2 \otimes \gamma_3)W^* = 1 \otimes \gamma_2.
			\end{align*}
			This transforms the product operator into
			\begin{equation*}
				D_\times \sim f(s) \otimes \gamma_2 \otimes \gamma_3 + \frac{1}{1+s}D^+_{ \S^{2k-1}} \otimes 1 \otimes \gamma_2 + i \partial_s \otimes \gamma_1 \otimes \gamma_3.
			\end{equation*}
			Finally, upon identifying $L^2( \S^{2k-1} \times (-\varepsilon, \varepsilon), \ \cS^+) \cong L^2( \S^{2k-1}, \ \cS^+) \otimes L^2((-\varepsilon, \varepsilon))$ we find that 
			\begin{align*}
				D_\times & \sim (D^+_{ \S^{2k-1}} \otimes \gamma_2) \otimes \frac{1}{1+s} + (1 \otimes \gamma_3) \otimes (\gamma_2 f(s) + i\gamma_1\partial_s), \nonumber \\
				& = \widetilde{D_{ \S^{2k-1}}} \otimes \frac{1}{1+s} + \gamma_3 \otimes T. \qedhere
			\end{align*}
		\end{proof}

		This expression for $D_\times$ is essential for our further investigation and in fact already 
                closely resembles the external product of $\widetilde{D^+_{ \S^{2k-1}}}$ or $D_{ \S^{2k}}$ and the index cycle $(L^2((-\varepsilon, \varepsilon), \C^2), T)$.
		Secondly, this separated form allows us to investigate the analytical properties of $D_\times$ in terms of the already understood operators $\widetilde{D_{ \S^{2k-1}}}$, $D_{ \S^{2k}}$ and $T$.
		
	\subsection{Analysis of the product operator}
	\label{sec:product_operator_analysis}
	
	We will now prove that $D_\times$ is essentially self-adjoint and that it 
        has compact resolvent.

		For self-adjointness we use the concept of an adequate approximate identity introduced in \cite{MeslandRennie2016}.
		The approach is similar to \cite{Dungen} where van den Dungen proves self-adjointness of a perturbed Dirac operator using an adequate approximate identity corresponding to the original Dirac operator. Let us recall the setup.
				
		\begin{definition}
			\label{def:adequate_approximate_identity}
			Let $D:\dom(D) \to H$ be a densely defined symmetric operator on some Hilbert space $H$.
			An \emph{adequate approximate identity} for $D$ is a sequential approximate identity $\{\phi_k\}_{k \in \N}$ on $H$ such that $\phi_k \dom(D^*) \subset \dom(\overline{D})$, $[\overline{D}, \phi_k]$ is bounded on $\dom(D)$ and $\sup_{k \in \N} \|\overline{[\overline{D}, \phi_k]}\| < \infty$.
		\end{definition}
		\begin{remark}
			The definition of an adequate approximate identity is usually given in the context of Hilbert modules.
			We restrict our attention to the Hilbert space case, since it suffices for our purposes.
			All results, such as Proposition \ref{prop:aai_implies_esa}, still hold in the Hilbert module case.
		\end{remark}
		
		The motivation for introducing these adequate approximate identities is the following proposition.
		
		\begin{proposition}
			\label{prop:aai_implies_esa}
			Let $D:\dom(D) \to H$ be a densely defined symmetric operator on a Hilbert space $H$ and suppose $\{\phi_k\}_{k \in \N}$ is an adequate approximate identity for $D$, then $D$ is essentially self-adjoint.
		\end{proposition}
		\begin{proof}
			See \cite{MeslandRennie2016}.
		\end{proof}
		
		We also have a converse.
		
		\begin{lemma}
			\label{lem:aai_of_sa_operator}
			Suppose $D:\dom(D) \to H$ is a self-adjoint operator.
			Then $\phi_k := (1 + \frac{1}{k^2}D^2)^{-1}$ defines an adequate approximate identity $\{\phi_k\}_{k \in \N}$ for $D$.
			Furthermore $\|(1 + \frac{1}{k^2}D^2)^{-1}\| \leq 1$ and $\|D(1 + \frac{1}{k^2}D^2)^{-1}\| \leq k$.
		\end{lemma}
		\begin{proof}
			The norm-estimates, as well as the fact that $\{\phi_k\}$ defines an approximate unit, are in \cite[Thm 5.1.9]{Pedersen}.
			Furthermore, this theorem tells us that $[D, \phi_k] = 0$ on $\dom(D)$.
			The only remaining requirement is then that $\phi_k \dom(D) \subset \dom(D)$, we even have the stronger result that $\phi_k H \subset \dom(D)$ since $\phi_k = k^2(D + ki)^{-1}(D - ki)^{-1}$ and the resolvents map $H$ into $\dom(D)$.
		\end{proof}
		
		We will show that, starting from adequate approximate identities for two self-adjoint operators $D_1$ and $D_2$, we can construct an adequate approximate identity for $D_1 \otimes A + B \otimes D_2$ provided we have some control over the interaction between $A$ and $D_2$, and $B$ and $D_1$.
		
		\begin{proposition}
			\label{prop:product_sum_is_esa}
			Let $D_1:\dom(D_1) \to H_1$ and $D_2:\dom(D_2) \to H_2$ be densely defined self-adjoint operators on Hilbert spaces $H_1$ and $H_2$.
			Let $A:H_2 \to H_2$ and $B:H_1 \to H_1$ be bounded, self-adjoint operators, such that $\|(1 + \frac{1}{k^2}D_1^2)^{-1}[B, D_1^2](1 + \frac{1}{k^2}D_1^2)^{-1}\| \leq c_1 k$ and $\|(1 + \frac{1}{k^2}D_2^2)^{-1}[A, D_2^2](1 + \frac{1}{k^2}D_2^2)^{-1}\| \leq c_2 k$ for some $c_1, c_2 \in \R$.
			Then $D_1 \otimes A + B \otimes D_2$ is essentially self-adjoint on $\dom(D_1) \algotimes \dom(D_2)$.
		\end{proposition}
		\begin{proof}
			We will show that
			\begin{align*}
				\phi_k := (1 + \frac{1}{k^2}D_1^2)^{-1} \otimes (1 + \frac{1}{k^2}D_2^2)^{-1}
			\end{align*}
			is an adequate approximate identity for $D_1 \otimes A + B \otimes D_2$, and then invoke Proposition \ref{prop:aai_implies_esa}.
			For ease of notation introduce $a = D_1 \otimes A + B \otimes D_2$.
			
			First note that $\phi_k$ is an approximate identity for $H_1 \otimes H_2$, since it clearly is one on the dense subspace $H_1 \algotimes H_2$.
			
			Next, we show that $\phi_k \dom(a^*) \subset \dom(\overline{a})$, in fact we will show the stronger $\phi_k H_1 \otimes H_2 \subset \dom(\overline{a})$ similar to what we saw in Lemma \ref{lem:aai_of_sa_operator}.
			We will use that $a\phi_k$ is bounded, indeed
			\begin{align*}
				\|a\phi_k \| 
				& \leq \| D_1(1 + \frac{1}{k^2}D_1^2)^{-1} \otimes A(1 + \frac{1}{k^2}D_2^2)^{-1}\| \\
				& \hspace{4em} + \|B(1 + \frac{1}{k^2}D_1^2)^{-1} \otimes D_2(1 + \frac{1}{k^2}D_2^2)^{-1} \|, \\
				& \leq k\|A\| + \|B\|k.
			\end{align*}
			
			Let $z \in H_1 \otimes H_2$, $z = \lim_n z_n$, $z_n \in H_1 \algotimes H_2$ and fix $k$.
			Clearly $\phi_k z_n \in \dom(a) = \dom(D_1) \algotimes \dom(D_2)$ and $\phi_k z_n \to \phi_k z$ since $\phi_k$ is bounded.
			Moreover, as we just saw, $a\phi_k$ is bounded so that $a\phi_k z_n \to a \phi_k z$.
			Since we have a sequence in $\dom(a)$ converging to $\phi_k z$ for which the images under $a$ also converge, we get $\phi_k z \in \dom(\overline{a})$.
			
			Finally we consider $[\overline{a}, \phi_k]$ on $\dom(a)$, and show that these commutators are bounded uniformly in $k$.
			Recall from the proof of Lemma \ref{lem:aai_of_sa_operator} that $(1 + \frac{1}{k^2}D_i^2)^{-1}$ and $D_i$ commute on $\dom(D_i)$.
			Then
			\begin{align*}
				[\overline{a}, \phi_k] & = D_1(1 + \frac{1}{k^2}D_1^2)^{-1} \otimes [A, (1 + \frac{1}{k^2}D_2^2)^{-1}]  \\ & \hspace{3em} + [B, (1 + \frac{1}{k^2}D_1^2)^{-1}] \otimes D_2(1 + \frac{1}{k^2}D_2^2)^{-1}.
			\end{align*}
			Since $\|D_i(1 + \frac{1}{k^2}D_i^2)^{-1}\| \leq k$ we want to find a bound of order $\frac{1}{k}$ for the commutators $[A, (1 + \frac{1}{k^2}D_2^2)^{-1}]$ and $[B, (1 + \frac{1}{k^2}D_1^2)^{-1}]$.
			
			We start by rewriting these commutators in terms of the original operators
			\begin{align*}
				[B, (1 + \frac{1}{k^2}D_1^2)^{-1}] & = (1 + \frac{1}{k^2}D_1^2)^{-1}[1 + \frac{1}{k^2}D_1^2, B](1 + \frac{1}{k^2}D_1^2)^{-1}, \\
				& = -\frac{1}{k^2} (1 + \frac{1}{k^2}D_1^2)^{-1}[B, D_1^2](1 + \frac{1}{k^2}D_1^2)^{-1}.
			\end{align*}
			By assumption there exists a $c$ such that
			\begin{align*}
				\|(1 + \frac{1}{k^2}D_1^2)^{-1}[B, D_1^2](1 + \frac{1}{k^2}D_1^2)^{-1}\| \leq c_1 k
			\end{align*}
			which implies
			\begin{align*}
				\|[B, (1 + \frac{1}{k^2}D_1^2)^{-1}]\| \leq c_1 \frac{1}{k}.
			\end{align*}
			By the same reasoning we get
			\begin{align*}
				\|[A, (1 + \frac{1}{k^2}D_2^2)^{-1}]\| \leq c_2 \frac{1}{k}.
			\end{align*}
			Together this implies $\|[\overline{a}, \phi_k]\| \leq c_1 + c_2$ which completes the proof.
		\end{proof}

		\begin{corollary}
			\label{cor:product_sum_essentially_self_adjoint_domains}
			If $D_1$ and $D_2$ are essentially self-adjoint on $\dom(D_1)$ and $\dom(D_2)$ and satisfy the assumptions in Proposition \ref{prop:product_sum_is_esa}, then $D_1\otimes A + B \otimes D_2$ is essentially self-adjoint on $\dom(D_1) \algotimes \dom(D_2)$.
		\end{corollary}
		\begin{proof}	
			Write $\dom(\overline{D_1})$ and $\dom(\overline{D_2})$ for the domains of self-adjointness of $D_1$ and $D_2$.
			Then we know that $D_1 \otimes A + B \otimes D_2$ is essentially self-adjoint on $\dom(\overline{D_1}) \algotimes \dom(\overline{D_2})$.
			Write $a_0$ for the closure of $D_1 \otimes A + B \otimes D_2$ defined on $\dom(D_1) \algotimes \dom(D_2)$ and $a$ for the closure on $\dom(\overline{D_1}) \algotimes \dom(\overline{D_2})$.
			
			Clearly $a_0 \subset a$, so we want to show that $a \subset a_0$.
			This follows if we can show that $\dom(\overline{D_1}) \algotimes \dom(\overline{D_2}) \subset \overline{\dom(D_1) \algotimes \dom(D_2)}$, with the closure taken in the graph-norm of $a$.
			So suppose $\psi \otimes \phi \in \dom(\overline{D_1}) \algotimes \dom(\overline{D_2})$.
			Then $\psi = \lim x_n$, $\phi = \lim y_n$ such that $D_1 \psi = \lim D_1 x_n$ and $D_2 \phi = \lim D_2 y_n$, with $x_n \in \dom(D_1)$ and $y_n \in \dom(D_2)$ since the $D_i$ are essentially self-adjoint on the $\dom(D_i)$.
			But then
			\begin{align*}
				\|a(x_n \otimes y_n) - & a(\psi \otimes \phi)\| \\ = & \| (D_1 \otimes A)(x_n \otimes y_n - \psi \otimes \phi) + (B \otimes D_2)(x_n \otimes y_n - \psi \otimes \phi) \|, \\
				\leq &  \|D_1 x_n \otimes A y_n - D_1 \psi \otimes A y_n\| + \|D_1 \psi \otimes A y_n - D_1 \psi \otimes A \phi\| \\
				& + \|B x_n \otimes D_2 y_n - B x_n \otimes D_2 \phi\| + \|B x_n \otimes D_2 \phi - B \psi \otimes D_2 \phi\|, \\
				\leq & \| D_1(x_n - \psi) \| \cdot \|A y_n\| + \|D_1 \psi \| \cdot \|A (y_n - \phi)\| \\
				& + \|B x_n\| \cdot \|D_2(y_n - \phi)\| + \| B(x_n - \psi) \|\cdot\|D_2 \phi \|
			\end{align*}
			tends to zero.
			Therefore $\psi \otimes \phi \in \overline{\dom(D_1) \algotimes \dom(D_2)}$ (closure in the graph norm) so that $a \subset a_0$.
		\end{proof}

		\begin{corollary}
			\label{cor:D_product_self_adjoint}
			The operator $D_\times$ is essentially self-adjoint on $\dom(\widetilde{D^+_{ \S^{2k-1}}}) \algotimes \dom(T)$ or $\dom(D_{ \S^{2k}}) \algotimes \dom(T)$ (depending on whether $n$ is odd or even).
		\end{corollary}
		\begin{proof}
			Referring to the notation of Proposition \ref{prop:product_sum_is_esa} we have $D_1 = \widetilde{D^+_{ \S^{2k-1}}}$ or $D_1 = D_{ \S^{2k}}$, $A = \frac{1}{1 + s}1_{\C^2}$, $B = \gamma_3$ and $D_2 = T$.
			
			The relevant commutators are
			\begin{align*}
				[A, D_2^2] = \left( \frac{2}{(1 + s)^3} + \frac{2}{(1 + s)^2}\partial_s \right)1_{\C^2}, \quad
				[B, D_1^2] = 0.
			\end{align*}
			
			We will prove that $\|\partial_s (1 + \frac{1}{k^2}T^2)^{-1}\| \leq k$ and use that to prove the required estimate.
			From Lemma \ref{lem:graph_norm_T} we know that for $\lambda = \pm\alpha$ we have $\|(T + \lambda i) \psi\| \geq \|\psi'\|$.
			This also holds for $|\lambda| \geq \alpha$ since $T$ is symmetric.
			In particular $\|(T + mi)\psi\| \geq \|\psi'\|$ for all $m \in \Z \setminus \{-1, 0, 1\}$ since $\alpha < 2$.
			Furthermore $\|(T + mi)\psi\| \geq m\|\psi\|$ for all $m \in \Z$ and $\psi \in \dom(T)$ by symmetry of $T$.
			
			Let $\phi \in L^2((-\varepsilon, \varepsilon), \C^2)$ be arbitrary.
			Since $T$ is self-adjoint, $T \pm ki$ are invertible so we may define $\psi = (1 + \frac{1}{k^2}T^2)^{-1} \phi$.
			If we combine the two estimates we have for $T \pm ki$ we get
			\begin{equation}
				\| \phi \| = \frac{1}{k^2}\| (T - ki)(T + ki) \psi \| \geq \frac{1}{k^2}k \|\psi'\|.
			\end{equation}

			Then
			\begin{equation}
				\|\partial_s (1 + \frac{1}{k^2}T^2)^{-1} \phi\| = \|\psi'\| \leq k \|\phi\|.
			\end{equation}
			
			So $\|\partial_s(1 + \frac{1}{k^2}T^2)^{-1}\| \leq k$, which in turn means that
			\begin{align*}
				\|(1 + \frac{1}{k^2}D_2^2)^{-1}[A, D_2^2](1 + \frac{1}{k^2}D_2^2)^{-1}\| & \leq 1 \cdot \left(\frac{2}{(1 - \varepsilon)^3}\cdot 1 + \frac{2}{(1 - \varepsilon)^2}\cdot k \right), \\
				& \leq \frac{4}{(1 - \varepsilon)^3}k.
			\end{align*}
			
			Therefore Proposition \ref{prop:product_sum_is_esa} applies, and $D_\times$ is essentially self-adjoint with the stated domains.
		\end{proof}

		\begin{remark}
			In \cite{KaadLeschUnbddProd} self-adjointness of the product operator is proven by showing that $D_1 \otimes 1$ and $\gamma \otimes_\nabla D_2$ separately are (essentially) self-adjoint and that they anti-commute, which then proves that their sum is again (essentially) self-adjoint.
			
			In our case $\gamma^3 \otimes_{\nabla^\E} D_{\R^{n+1}}$ is not essentially self-adjoint on the domain $C^\infty_0( \S^{2k-1} \times (-\varepsilon, \varepsilon))$ (with the appropriate unitary transformations and spinor components), which should be the domain according to \cite{KaadLeschUnbddProd}, so their results on using connections are not directly applicable.
		\end{remark}
		
		Now that we have self-adjointness of $D_\times$, we turn to the resolvents of $D_\times$.
		To prove that these resolvents are compact we will use the min-max principle.
		
		\begin{proposition}[min-max principle]
			\label{prop:min_max_principle}
			Let $D:\dom(D) \to H$ be a self-adjoint operator that is bounded below.
			Then $D$ has compact resolvent if and only if $\mu_n(D) \to \infty$ as $n \to \infty$, where
			\begin{align*}
				\mu_n(D) & = \sup_{\phi_1, ..., \phi_{n-1}} U_D(\phi_1, ..., \phi_{n-1}), \\
				U_D(\phi_1, ..., \phi_m) & = \inf_{\psi \in \dom(D), \|\psi\| = 1, \psi \perp \phi_k \forall k} \langle \psi, A\psi \rangle.
			\end{align*}
		\end{proposition}
		\begin{proof}
			See Theorems XIII.1 and XIII.64 in \cite{ReedSimon}.
		\end{proof}
				We will also use the following characterization of compact resolvents.
		
		\begin{proposition}
			\label{prop:compact_resolvent_gives_eigenvectors}
			Let $D:\dom(D) \to H$ be a self-adjoint operator that is bounded below.
			Then $D$ has compact resolvents if and only if there exists a complete orthonormal basis $\{\phi_n\}_{n \in \N}$, $\phi_n \in \dom(D)$ for $H$ consisting of eigenvectors for $D$ with eigenvalues $\lambda_1 \leq \lambda_2 \leq ... $ and $\lambda_n \to \infty$.
		\end{proposition}
		\begin{proof}
			See Theorem XIII.64 in \cite{ReedSimon}.
		\end{proof}
		
		We will apply the min-max principle to $D_\times^2$, which is positive, and hence bounded below, since it is the square of a self-adjoint operator.
		Note that if $D_\times^2$ has compact resolvent, so does $D_\times$.
		
		\begin{proposition}
			\label{prop:D_product_compact_resolvent}
			The operator $D_\times^2$ has compact resolvents.
		\end{proposition}
		\begin{proof}
			We will show that $\mu_n(D_\times) \to \infty$ and invoke Proposition \ref{prop:min_max_principle}.
			
			Recall that for odd-dimensional spheres we have $D_\times = \widetilde{D^+_{ \S^{2k-1}}} \otimes \frac{1}{1+s} + \gamma_3 \otimes T$, while for even-dimensional spheres we have $D_\times = D_{ \S^{2k}} \otimes \frac{1}{1+s} + \gamma_r \otimes T$.
			In both cases we have the same structure that can be characterized as $D_\times = D \otimes \frac{1}{1+s} + \Gamma \otimes T$.
			It is this structure that enables the following proof, for simplicity we will prove it in the odd case. 
			
                        First of all, note that the square of $D_\times$ is 
			\begin{equation*}
				D_\times^2 = (\widetilde{D^+_{ \S^{2k-1}}})^2 + \gamma_3 \widetilde{D^+_{ \S^{2k-1}}} \otimes -i \gamma_1 \frac{1}{(1+s)^2} + 1 \otimes T^2.
			\end{equation*}
			Since $T$ has compact resolvent, so does $T^2$ which means that by Proposition \ref{prop:compact_resolvent_gives_eigenvectors} there is a complete orthonormal basis of eigenvectors $ \{ \psi_n \}_{n \in \N} \subset L^2((-\varepsilon, \varepsilon), \C^2)$ for $T^2$.
			These eigenvectors can easily be adapted to eigenvectors for $T$ with eigenvalues $\lambda_n$ such that $\lambda_n^2$ is an increasing sequence tending to infinity.
			Similarly we get a complete orthonormal basis of eigenvectors $\{\phi_n\}_{n \in \N} \subset L^2( \S^{2k-1}, \ \cS^+ \otimes \C^2)$ for $\widetilde{D^+_{ \S^{2k-1}}}$ with eigenvalues $\nu_n$ such that $\nu_n^2$ is increasing and unbounded.
			
			The set $\{ \phi_k \otimes \psi_l \}_{(k, l) \in \N \times \N}$ is a complete orthonormal set for $L^2( \S^{2k-1}, \C^2) \otimes L^2((-\varepsilon, \varepsilon), \C^2)$, using this set we will show that $\mu_n(D_\times^2) \to \infty$.
			It is clear from the definition that the $\mu_n(D^2_\times)$ form an increasing sequence in $n$, so it is sufficient to show that $\mu_{n^2 + 1}(D_\times^2) \to \infty$.
			
			Fix $n \in \N$, we will compute a lower bound for $U_{D_\times^2}( \{ \phi_k \otimes \psi_l \}_{1 \leq k, l \leq n})$, which in turns gives a lower bound for $\mu_{n^2+1}(D_\times^2)$.
			Since $\{ \phi_k \otimes \psi_l \}_{(k, l) \in \N \times \N}$ is a complete set any element of $\dom(D_\times^2)$ is a limit of a sequence of finite linear combinations of the $\phi_k \otimes \psi_l$.
			This leads us to consider
			\begin{align*}
				\langle \phi_k \otimes \psi_l, D_\times^2 (\phi_k \otimes \psi_l) \rangle 
				= & \langle \phi_k \otimes \psi_l, (1 \otimes T^2)\phi_k \otimes \psi_l \rangle  \\ & + \langle \phi_k \otimes \psi_l, \left(\left(\widetilde{D^+_{ \S^{2k-1}}} \otimes \frac{1}{(1+s)^2}^2\right) \right) \phi_k \otimes \psi_l \rangle \\
				  & + \langle \phi_k \otimes \psi_l, \left(\gamma_3 \widetilde{D^+_{ \S^{2k-1}}} \otimes  -i\gamma_1 \frac{1}{(1 + s)^2} \right)\phi_k \otimes \psi_l \rangle, \\
				= & \lambda_l^2 + \nu_k^2 \left\| \frac{1}{1+s} \psi_l \right\|^2 - i\langle \phi_k, \gamma_3 \widetilde{D^+_{ \S^{2k-1}}} \phi_k \rangle\langle \psi_l, \gamma_1 \frac{1}{(1 + s)^2} \psi_l \rangle, \\
				= & \lambda_l^2 + \nu_k^2 \left\| \frac{1}{1+s} \psi_l \right\|^2.
			\end{align*}
			The cross-term vanishes because $\gamma_3$ and $\widetilde{D^+_{ \S^{2k-1}}}$ anti-commute.
			Since $\frac{1}{1 + s}$ is bounded below by $\frac{1}{1 + \varepsilon}$ and above by $\frac{1}{1-\varepsilon}$ on $(-\varepsilon, \varepsilon)$ we find that
			\begin{align*}
				\langle \phi_k \otimes \psi_l, D_\times^2 \phi_k \otimes \psi_l \rangle \geq \frac{1}{(1 + \varepsilon)^2} \nu_k^2 + \lambda_l^2.
			\end{align*}
	
			Every element $\Psi$ of $\dom(D_\times^2)$ can be written
			\begin{align*}
				\Psi = \sum_{k,l=0}^{\infty} \alpha_{(k, l)} \phi_k \otimes \psi_l
			\end{align*}
			since the $\{\phi_k \otimes \psi_l\}$ form a complete orthonormal set.
			
			If $\Psi$ is an admissible element in the infimum of $U_{D_\times^2}(\{\phi_k \otimes \psi_l\}_{1 \leq k, l \leq n})$, then $\alpha_{(k, l)} = 0$ for $k, l \leq n$ and $\sum |\alpha_{(k, l)}|^2 = 1$, which means
			\begin{align*}
				\langle \Psi, D_\times^2 \Psi \rangle \geq \frac{1}{(1+\varepsilon)^2}\nu_{n+1}^2 + \lambda_{n+1}^2.
			\end{align*}
			The right hand side of this equation clearly tends to infinity as $n$ tends to infinity, so $\mu_{n^2+1}(D_\times^2)$ tends to $\infty$ as desired.
		\end{proof}

	\subsection{Relation to the Kasparov product of \texorpdfstring{$\imath_!$}{i!} and \texorpdfstring{$[\R^{n+1}]$}{R n+1}}
	\label{sec:kasparov_product}
		
		Now that we have established the analytical properties of $D_\times$ it is time to turn to our primary goal and establish the unbounded factorization of $[\S^n]$ as the product of the unbounded shriek cycle and Euclidean space. This also provides, in a sense, a factorization of $D_{ \S^n}$ as a product of $S$ and $D_{\R^{n+1}}$, although we are left with the explicit remainder $T$, that becomes trivial in bounded $KK$-theory.
		
		\begin{theorem}
			Let $n \geq 1$. Then
			\begin{description}
				\item[$n$ odd]
				The data $(\tilde{\E} \otimes_{C_0(\R^{n+1})} L^2(\R^{n+1}, \cS), D_\times; \gamma^3 \otimes \gamma^3)$ defines an unbounded Kasparov $C( \S^n) \otimes \Cl_1$-$\C$ cycle that represents both $\imath_! \otimes [\R^{n+1}]$ and $\widetilde{[ \S^n]} \otimes \1$ in $KK_0(C( \S^n) \otimes \Cl_1, \C)$.
	
				\item[$n$ even]
				The data $(\tilde{\E} \otimes_{C_0(\R^{n+1}) \otimes \Cl_1} L^2(\R^{n+1}, \cS), D_\times; \gamma^3 \otimes \gamma_r)$ defines an unbounded Kasparov $C( \S^n)$-$\C$ cycle that represents both $\imath_! \otimes \widetilde{[\R^{n+1}]}$ and $[ \S^n] \otimes \1$ in $KK_0(C( \S^n), \C)$.
			\end{description}
		\end{theorem}
		
		\begin{proof}
			Again we will do the proof in the case $n$ odd, however the same strategy works for the case where $n$ is even.
		
			We have proven that $D_\times$ is self-adjoint and has compact resolvent in Section \ref{sec:product_operator_analysis}.
			Moreover, the commutators of $D_\times$ with $C^1( \S^{2k-1}) \otimes \Cl_1$ are bounded so $(\tilde{\E} \otimes_{C_0(\R^{2k})} L^2(\R^{2k}, \cS), D_\times; \gamma_3 \otimes \gamma_3)$ is an unbounded Kasparov cycle.
			
			The remainder of the proof deals with verifying Kucerovsky's criterion \cite[Theorem 13]{Kucerovsky} in both cases.
			In the case $\imath_! \otimes [\R^{2k}]$ we will use the expression in Equation \ref{eq:product_operator_definition}, and in the $\widetilde{[ \S^{2k-1}]}\otimes \1$ case we use the expression in Equation \ref{eq:product_operator_radial_split}.


Let us first consider Kucerovsky's connection condition for the product $\imath_! \otimes_{C_0(\R^{2k})} [\R^{2k}]$ where a general computation using the properties of a metric connection suffices. Indeed, let $\xi \in \tilde{\E}$ be a homogeneous element of degree $\deg \xi$, and define $T_\xi:L^2(\R^{2k}, \cS) \to \tilde{\E} \otimes_{C_0(\R^{2k})} L^2(\R^{2k}, \cS)$ by $\psi \mapsto \xi \otimes \psi$.
			The adjoint is given by $\phi \otimes \psi \mapsto \langle \xi, \phi \rangle_{\tilde{\E}} \cdot \psi$ for an elementary tensor $\phi \otimes \psi \in \tilde{\E} \otimes_{C_0(\R^{2k})} L^2(\R^{2k}, \cS)$.
			The connection condition for the product $\imath_! \otimes [\R^{2k}]$ is, in this case, that the graded commutator
			\begin{equation*}
				\left[ \matTwo{D_\times}{0}{0}{D_{\R^{2k}}}, \matTwo{0}{T_\xi}{T_\xi^*}{0} \right]
			\end{equation*}
			is bounded for $\xi$ in a dense subset of $\E$.
			
			A simple calculation shows that this is equivalent to boundedness of
			\begin{equation*}
				\left[ \matTwo{\gamma^3 \otimes_{\nabla^{\tilde{\E}}} D_{\R^{2k}}}{0}{0}{D_{\R^{2k}}}, \matTwo{0}{T_\xi}{T_\xi^*}{0} \right]
			\end{equation*}
			
			Evaluating the bottom-left component of the resulting matrix on the elementary tensor $\phi \otimes \psi \in \tilde{\E} \otimes_{C_0(\R^{2k})} L^2(\R^{2k}, \cS)$ yields
			\begin{align*}
				D_{\R^{2k}} T_\xi^* - (-1)^{\deg \xi} & T_\xi^* (\gamma^3 \otimes_{\nabla^{\tilde{\E}}} D_{\R^{2k}}) = \\
				& = D_{\R^{2k}} (\langle \xi, \phi \rangle_{\tilde{\E}} \cdot \psi) - (-1)^{\deg \xi} T_\xi^*(\gamma^3 \phi \otimes D_{\R^{2k}} \psi + \nabla^{\tilde{\E}}(\gamma^3 \phi) \cdot \psi), \\
				& = D_{\R^{2k}} (\langle \xi, \phi \rangle_{\tilde{\E}} \cdot \psi) - \langle \xi, \phi \rangle_{\tilde{\E}} \cdot D_{\R^{2k}} \psi - \langle \xi, \nabla^{\tilde{\E}}(\phi) \rangle_{\tilde{\E}} \cdot \psi, \\
				& = [D_{\R^{2k}}, \langle \xi, \phi \rangle_{\tilde{\E}} ] \psi - \langle \xi, \nabla^{\tilde{\E}}(\phi) \rangle_{\tilde{\E}} \cdot \psi, \\
				& = \langle \nabla^{\tilde{\E}}(\xi), \phi \rangle_{\tilde{\E}} \cdot \psi.
			\end{align*}
			This is bounded by $||\nabla^{\tilde{\E}}(\xi)||$ which is indeed finite for a dense subset of $\tilde{\E}$.
			Here $\langle \nabla^{\tilde{\E}}(\xi), \phi \rangle_{\tilde{\E}}$ acts on $L^2(\R^{2k}, \cS)$ in the way described in Lemma \ref{lem:connection}.
			
			The top-right component is bounded by a similar computation, the diagonal components are 0.
			This computation is general for metric connections, in fact, whenever a product operator is constructed using a metric connection, the connection condition is automatically satisfied.

			The compatibility condition is straightforward, simply by taking the domain of compatibility to be $\W = C^\infty_c( \S^{2k-1} \times (-\varepsilon, \varepsilon))$ (embedded appropriately in the respective spaces).
			
We then consider the positivity condition. Using symmetry of $\tilde{S}$ and $D_\times$ we find that we need to prove that
			\begin{align*}
				\langle \psi, ((\tilde{S} \otimes 1)D_\times + D_\times(\tilde{S} \otimes 1))\psi \rangle \geq C \langle \psi, \psi \rangle,
			\end{align*}
			holds on $C_c^\infty( \S^{2k-1} \times (-\varepsilon, \varepsilon)) \otimes \C^2 \otimes \C^2$ for some $C \in \R$.
			Using the (anti)-commutation properties of the $\gamma$-matrices, we find that
			\begin{align*}
				\langle \psi, ((\tilde{S} \otimes 1)D + D(\tilde{S} \otimes 1))\psi \rangle
				& = \langle \psi, (2 f(s)^2 + f'(s) \otimes \gamma_1 \otimes \gamma_1) \psi \rangle, \\
				& = \langle \psi, f(s)^2 \psi \rangle + \langle \psi, (f(s)^2 + f'(s) \otimes \gamma_1 \otimes \gamma_1) \psi \rangle, \\
				& \geq \langle \psi, (f(s)^2 - f'(s)) \psi \rangle, \\
				& = - \alpha^2 \langle \psi, \psi \rangle,
			\end{align*}
			so we may choose $C = - \alpha^2$.

			Let us now turn to the product $\widetilde{[ \S^{2k-1}]}\otimes\1$. The connection condition requires a more explicit computation. To avoid notational confusion between the maps $T_\xi$ and the operator $T$ we write $D_2$ for $T$ in this computation, similar to the notation in \cite{Kucerovsky}.
			In this case the connection condition is that the commutator
			\begin{align*}
				\left[ \matTwo{D_\times}{0}{0}{D_2}, \matTwo{0}{T_\xi}{T_\xi^*}{0} \right] = \matTwo{0}{D_\times T_\xi - (-1)^{\deg \xi} T_\xi D_2}{D_2 T_\xi^* - (-1)^{\deg \xi} T_\xi^* D_\times}{0}
			\end{align*}
			is bounded for $\xi \in C^1( \S^{2k-1}, \C^2)$.
			
			As a first step, note that $T_\xi D_2 - (1 \otimes D_2)T_\xi = 0$ and that $D_2 T_\xi^* - T_\xi^* (1 \otimes D_2) = 0$.
			The grading-factors introduced by the commutator cancel against the $\gamma^3$ appearing in $D_\times$ so that the connection condition reduces to
			\begin{align*}
				\left[ \matTwo{\widetilde{D^+_{ \S^{2k-1}}} \otimes \frac{1}{1+s}}{0}{0}{0}, \matTwo{0}{T_\xi}{T_\xi^*}{0} \right] = \matTwo{0}{\left(\widetilde{D^+_{ \S^{2k-1}}} \otimes \frac{1}{1+s}\right) T_\xi}{- (-1)^{\deg \xi} T_\xi^* \left(\widetilde{D^+_{ \S^{2k-1}}} \otimes \frac{1}{1+s}\right)}{0}
			\end{align*}
			Using the self-adjointness of $\widetilde{D^+_{ \S^{2k-1}}}$ in the bottom-left, this equals
			\begin{equation}
				\frac{1}{1+s} \matTwo{0}{T_{\widetilde{D^+_{ \S^{2k-1}}} \xi}}{-(-1)^{\deg \xi} T^*_{\widetilde{D^+_{ \S^{2k-1}}}\xi}}{0}
			\end{equation}
			which is indeed bounded for $\xi \in C^1( \S^{2k-1}, \C^2)$. 

The compatibility condition is again straightforward, while the positivity condition amounts to showing that 

			\begin{align*}
				\langle \phi \otimes \psi, ((\widetilde{D^+_{ \S^{2k-1}}}\otimes 1) D_\times + D_\times (\widetilde{D^+_{ \S^{2k-1}}} \otimes 1))(\phi \otimes \psi) \rangle \geq C \langle \phi \otimes \psi, \phi \otimes \psi \rangle
			\end{align*}
                        for some $C \in \R$. Since $\widetilde{D^+_{ \S^{2k-1}}} \otimes 1$ anti-commutes with $\gamma_3 \otimes T$ this term drops out, while $\widetilde{D^+_{ \S^{2k-1}}} \otimes 1$ commutes with $\widetilde{D^+_{ \S^{2k-1}}} \otimes \frac{1}{1+s}$ to give
			\begin{align*}
				\langle \phi \otimes \psi, ((\widetilde{D^+_{ \S^{2k-1}}}\otimes 1) D_\times + & D_\times (\widetilde{D^+_{ \S^{2k-1}}} \otimes 1))(\phi \otimes \psi) \rangle \\ 
				& = 2\langle \phi \otimes \psi, \left(\widetilde{D^+_{ \S^{2k-1}}}\right)^2 \phi \otimes \frac{1}{1+s} \psi \rangle, \\
				& \geq \frac{2}{1-\varepsilon} ||\widetilde{D^+_{ \S^{2k-1}}} \phi||^2 ||\psi||^2 \geq 0. \qedhere
			\end{align*}
		\end{proof}

\appendix
\section{Unbounded KK-cycles: from odd to even}
\label{sect:app}
		At several points in this paper we need to distinguish between the case $n$ even and $n$ odd.
		The fundamental class of a manifold $M$ with $\dim(M)$ even will yield an even unbounded Kasparov cycle, while if $\dim(M)$ is odd we get an odd unbounded Kasparov cycle.
		However, we want to work with even cycles exclusively, since that is where Kucerovsky's criterion is applicable.
We accomplish this by using the isomorphisms
                $KK_0(A \otimes \Cl_1, B) \cong KK_1(A, B) \cong KK_0(A, B \otimes \Cl_1)$, which at the level of concrete cycles are given by the following lemma.
		
		\begin{lemma}
			Let $(\E, D)$ be an odd unbounded Kasparov $A$-$B$ cycle.
			Then
			\begin{enumerate}
				\item $(\E \otimes \C^2, D \otimes \gamma^2; 1 \otimes \gamma^3)$ is an even unbounded Kasparov $A\otimes \Cl_1$-$B$ cycle, with $\Cl_1$ acting by $1 \otimes \gamma^1$.
				We call this the left-doubling of $(\E, D)$.
				\item $(\E \otimes \C^2, D \otimes \gamma^1; 1 \otimes \gamma^3)$ is an even unbounded Kasparov $A$-$B \otimes \Cl_1$ cycle with $\Cl_1$ acting by $1 \otimes \gamma^1$.
				We call this the right-doubling of $(\E, D)$.
			\end{enumerate}
			
			Conversely, any even $A \otimes \Cl_1$-$B$ cycle is equivalent to the left-doubling of an odd $A$-$B$ cycle in $KK_0(A \otimes \Cl_1, B)$ and any $A$-$B \otimes \Cl_1$ cycle is the right-doubling of the positive eigenspace of the non-trivial generator of $\Cl_1$.
			\label{lem:doubling_cycles}
		\end{lemma}
		\begin{proof}
			The only interesting claim in this Lemma is that every even $A \otimes \Cl_1$-$B$ corresponds to an odd $A$-$B$ cycle, since this requires the equivalence relations of $KK$-theory.
			The difficulty in this ``halving'' procedure is that the operator might not anti-commute with the action of $\Cl_1$ as in the case of a doubled odd cycle.
			In \cite[Thm. 5.1]{Dungen} van den Dungen shows that the operator can be modified such that it does, without changing the represented $KK$-class.
		\end{proof}

\end{document}